\documentclass[11pt]{article}
\usepackage{amsmath}
\usepackage{amssymb}
\usepackage{amsthm}
\usepackage[usenames]{color}
\usepackage{amscd}
\usepackage{dsfont}
\usepackage{indentfirst}

\usepackage[colorlinks=true,linkcolor=blue,filecolor=red,
citecolor=webgreen]{hyperref}
\definecolor{webgreen}{rgb}{0,.5,0}

\def\C{{\mathds{C}}}

\def\N{{\mathds{N}}}
\def\Z{{\mathds{Z}}}

\def\1{{\bf 1}}

\def\id{\operatorname{id}}

\newtheorem{theorem}{Theorem}[section]

\newtheorem{lemma}[theorem]{Lemma}
\newtheorem{cor}[theorem]{Corollary}

\begin{document}

\title{{\bf Estimates for $k$-dimensional spherical summations of arithmetic functions of the GCD and LCM}}
\author{Randell Heyman and L\'aszl\'o T\'oth}
\date{}
\maketitle

\centerline{H. Maier et al. (eds.) Number Theory in Memory of Eduard Wirsing}

\centerline{Springer, Cham, 2023, pp. 157--183}

\centerline{\url{https://doi.org/10.1007/978-3-031-31617-3_11}}

\begin{abstract} Let $k\ge 2$ be a fixed integer. We consider sums of type $\sum_{n_1^2+\cdots+ n_k^2\le x} F(n_1,\ldots,n_k)$,
taken over the $k$-dimensional spherical region $\{(n_1,\ldots,n_k)\in \Z^k: n_1^2+\cdots+ n_k^2\le x\}$, where $F:\Z^k\to \C$ is a 
given function. In particular, we deduce asymptotic formulas with remainder terms for the spherical
summations $\sum_{n_1^2+\cdots+ n_k^2\le x} f((n_1,\ldots,n_k))$ and $\sum_{n_1^2+\cdots+ n_k^2\le x} f([n_1,\ldots,n_k])$,
involving the GCD and LCM of the integers $n_1,\ldots,n_k$, where $f:\N\to \C$ belongs to certain classes of functions.
\end{abstract}



\tableofcontents

\section{Introduction}

For a function $F:\Z^k\to \C$ of $k$ ($k\ge 2$) variables consider the summation
\begin{equation} \label{sum_F}
\sum_{\substack{n_1,\ldots,n_k\in \Z\\ n_1^2+\cdots +n_k^2 \le x}} F(n_1,\ldots,n_k)
\end{equation}
over the lattice points in the $k$-dimensional sphere of radius $\sqrt{x}$.

If $F$ is the constant $1$ function, then \eqref{sum_F} represents the number of integer lattice points in the $k$-dimensional sphere of 
radius $\sqrt{x}$, and it has been extensively studied in the literature. See Section \ref{Sect_Known_estimates}. If $F(n_1,\ldots, n_k)=1$ for 
$(n_1,\ldots,n_k)=1$ and $0$ otherwise, then \eqref{sum_F} gives the number of primitive lattice points, much studied in the cases 
$k=2$ (primitive circle problem) and $k=3$. See, e.g., Chamizo et al. \cite{CCU2007}, Wu \cite{Wu2002}. 

There are in the literature only few asymptotic results concerning the sums \eqref{sum_F} for general classes of functions or for other 
special functions $F$. General results concerning classes of (multiplicative) multivariable functions $F$, with summation over 
$n_1,\ldots,n_k \in \N$  with $n_1,\ldots,n_k\le x$ have been given by Essouabri et al. \cite{EST2022}, de~la~Bret\`{e}che \cite{Bre2001}, 
Ushiroya \cite{Ush2012}. Also see the survey paper by the second author \cite{Tot2014}. Results for sums of type \eqref{sum_F}, 
more generally for sums over $n_1,\ldots,n_k\in \N$ with the H\"older norm 
$(n_1^s+\cdots +n_k^s)^{1/s}\le x$, where $s\ge 1$ is a real number, were considered in \cite{EST2022}. 
For example, as a consequence of certain more general results, it was proved in \cite[Cor.\ 5]{EST2022} by analytic methods that for every
$k\ge 2$,
\begin{equation*}
\sum_{\substack{n_1,\ldots,n_k\in \N\\ (n_1^s+\cdots +n_k^s)^{1/s} \le x}} c_k(n_1,\ldots,n_k) = x^k Q(\log x) + O(x^{k-\beta}),
\end{equation*}
where $c_k(n_1,\ldots,n_k)$ denotes the number of cyclic subgroups of the group $\Z_{n_1}\times \cdots \times \Z_{n_k}$, $Q(t)$ is a 
polynomial in $t$ of degree $2^k-1$ and $\beta$ is a positive real number.

Another similar recent result is, in the case $k=2$,
\begin{equation*} 
\sum_{\substack{n_1,n_2\in \N\\ n_1^2+n_2^2 \le x}} s(n_1,n_2) = x P(\log x) + O(x^{17/22+\varepsilon}),
\end{equation*}
where $s(n_1,n_2)$ stands for the total number of subgroups of the group $\Z_{n_1}\times \Z_{n_2}$ and $P(t)$ is a 
polynomial in $t$ of degree three. See Sui and Liu \cite{SL2021}.

In this paper we first deduce a $k$-dimensional generalization of Wintner's mean value theorem for the spherical summation \eqref{sum_F}. 
Then we consider the $k$-dimensional spherical summations of arithmetic functions $f$ 
of the GCD and LCM, given by
\begin{equation} \label{sum_GCD} 
\sum_{\substack{n_1,\ldots,n_k\in \Z\\ n_1^2+\cdots +n_k^2 \le x}} 
f((n_1,\ldots,n_k)) 
\end{equation}
and
\begin{equation} \label{sum_LCM}
\sum_{\substack{n_1,\ldots,n_k\in \Z\\ n_1^2+\cdots +n_k^2 \le x}} 
f([n_1,\ldots,n_k]).
\end{equation} 

For every function $f:\N \cup \{0\} \to \C$ with $f(0)=0$ one has 
\begin{equation} \label{sum_GCD_formula}
\sum_{\substack{n_1,\ldots,n_k\in \Z\\ n_1^2+\cdots +n_k^2 \le x}} 
f((n_1,\ldots,n_k)) 
= \sum_{\substack{d,\delta\in \N\\ d^2\delta \le x}} (\mu*f)(d) r_k(\delta),
\end{equation}
where $\mu$ is the M\"obius function, $*$ is the convolution of arithmetic functions, and $r_k(n)$ denotes
the number of $k$-tuples $(n_1,\ldots,n_k) \in \Z^k$ such that $n_1^2+\cdots + n_k^2=n$. See Lemma \ref{Lemma_gcd}.
Therefore, estimates involving the GCD of integers, are consequences of known estimates on the sums 
$\sum_{n\le x} r_k(n)$. There is no simple formula, similar to \eqref{sum_GCD_formula}, concerning the LCM of integers, 
and to establish asymptotic formulas for the sums \eqref{sum_LCM} is more difficult. 

We present asymptotic formulas for the sums \eqref{sum_GCD} and \eqref{sum_LCM}  
for large classes of functions $f$, including the functions $f=\id, \tau, \sigma, \varphi,\omega$, $\Omega$, and  
$f=\mu^2, \id, \sigma, \varphi$, respectively. Some of our results can be extended to H\"older norms 
$(n_1^s+\cdots +n_k^s)^{1/s}$, with $s\ge 1$, but we confine ourselves with the case $s=2$. All our proofs are elementary and 
different from the proofs of the above mentioned papers. 

For some other results on sums of arithmetic functions of the GCD and LCM see, e.g., Bordell\`es and T\'oth \cite{BorTot2022}, 
the authors \cite{HeyTot2021, HeyTot2022}, Hilberdink and T\'oth  \cite{HilTot2016}, Hilberdink et al. \cite{HLT2020}, 
T\'oth and Zhai \cite{TotZha2018}, and their references.  

Throughout the paper we use the following notation: $\N =\{1,2,\ldots\}$; $(n_1,\ldots,n_k)$ and $[n_1,\ldots,n_k]$ denote the greatest common
divisor (GCD) and least common multiple (LCM) of $n_1,\ldots,n_k\in \Z$; $*$ is the Dirichlet
convolution of arithmetic functions; $\1$, $\delta$ and $\id$ are the functions $\1(n)=1$, $\delta(n)=\lfloor 1/n \rfloor$, $\id(n)=n$ ($n\in \N$); 
$\mu$ denotes the M\"obius function; $\varphi$ is Euler's totient function; $\lambda$ is the Liouville function; 
$\tau(n)$ and $\sigma(n)$ are the number and sum of divisors of $n\in \N$;   
$\omega(n)$ and $\Omega(n)$ stand for the number of prime divisors, respectively prime power divisors of $n\in \N$;
the sums $\sum_p$ and products $\prod_p$ are taken over the primes $p$.

\section{Preliminaries}

\subsection{Known estimates for the number of lattice points in $k$-dimensional spheres} 
\label{Sect_Known_estimates}

For $k\ge 2$ let $r_k(n)$ denote, as mentioned above, the number of $k$-tuples $(n_1,\ldots,n_k) \in \Z^k$ such that $n=n_1^2+\cdots + n_k^2$. 
It is known that
\begin{equation} \label{number_lattice_points_k}
\sum_{\substack{n_1,\ldots,n_k\in \Z\\ n_1^2+\cdots +n_k^2\le x}} 1 = \sum_{0\le n\le x} r_k(n)= 
V_k x^{k/2}+P_k(x),
\end{equation}
where
\begin{equation} \label{V_k}
V_k= \frac{\pi^{k/2}}{\Gamma(k/2+1)}
\end{equation}
is the volume of the $k$-dimensional unit sphere ($\Gamma$ is the Gamma function), more explicitly, 
\begin{equation*}
V_{2m}= \pi^m/m!, \quad V_{2m+1}= 2^{m+1}\pi^m/(2m+1)!! \quad (m\ge 1),
\end{equation*}
with the notation $(2m+1)!!= 1\cdot 3\cdot 5\cdots (2m+1)$, and the best results up to date for the error term $P_k(x)$ are 
$P_2(x)\ll x^{517/1648+\varepsilon}$, where $517/1648 \approx 0.313713$ (J.~Bourgain and N.~Watt, 2017) 
$P_3(x) \ll x^{21/32+\varepsilon}$, where $21/32 = 0.65625$ (D.~R.~Heath-Brown, 1999), 
$P_4(x)\ll x(\log x)^{2/3}$ (A.~Walfisz, 1960), $P_k(x)\ll x^{k/2-1}$ if $k\ge 5$. 

We remark that in the cases $k=2$ and $k=3$ the problem of the best possible error terms is still unsolved (circle and spherical problems). See, e.g., Berndt et al.
\cite{BKZ2018}, Bourgain and Watt \cite{BouWat2017}, Grosswald \cite{Gro1985}, Ivi\'{c} et al. \cite{Ivi_etal2006}, Kr\"atzel \cite[Chs.\ 3 and 4]{Kra1988}, 
Sierpi\'nski \cite[Ch.\ XI]{Sie1988} for more details, 
including methods of proofs (elementary and analytic). 

\subsection{Sums of functions of the GCD and LCM over the integers, respectively natural numbers}

For the GCD and LCM we use the conventions $(0,0)=[0,0]=0$, $(n_1,\ldots,n_{k-1},0) = (n_1,\ldots,n_{k-1})$, $[n_1,\ldots,n_{k-1},0] =0$, and  
$(n_1)= [n_1]= |n_1|$ ($n_1,\ldots, n_{k-1}\in \Z)$. In this way, and by the symmetry of the LCM,
\begin{equation*} 
\sum_{\substack{n_1,\ldots,n_k\in \Z\\ n_1^2+\cdots +n_k^2 \le x}}
f([n_1,\ldots,n_k]) = 2^k \sum_{\substack{n_1,\ldots,n_k\in \N\\ n_1^2+\cdots +n_k^2 \le x}}
f([n_1,\ldots,n_k])
\end{equation*}
for every function $f:\N \cup \{0\}\to \C$,
hence LCM estimates over natural numbers are equivalent to estimates for the integers. 

For the GCD sums,
\begin{equation*} 
\sum_{\substack{n_1,n_2\in \Z\\ n_1^2+n_2^2\le x}}
f((n_1,n_2)) = 4 \sum_{\substack{n_1,n_2\in \N \\ n_1^2+n_2^2 \le x}}
f((n_1,n_2))+ 4 \sum_{\substack{n\in \N\\ n\le \sqrt{x}}} f(n),
\end{equation*}
and
\begin{equation*} 
\sum_{\substack{n_1,n_2\in \N\\ n_1^2+n_2^2\le x}}
f((n_1,n_2)) = \frac1{4} \sum_{\substack{n_1,n_2\in \Z \\ n_1^2+n_2^2 \le x}}
f((n_1,n_2)) -\frac1{2} \sum_{\substack{n\in \Z\\ n^2\le x}} f(|n|).
\end{equation*}

In general, in $k$-dimension ($k\ge 2$) we have
\begin{equation} \label{Z_N} 
\sum_{\substack{n_1,\ldots, n_k\in \Z\\ n_1^2+\cdots+ n_k^2\le x}}
f((n_1,\ldots, n_k)) = \sum_{j=0}^{k-1} \binom{k}{j} 2^{k-j} \sum_{\substack{n_1,\ldots, n_{k-j}\in \N \\ n_1^2+\cdots+ n_{k-j}^2 \le x}}
f((n_1,\ldots, n_{k-j})),
\end{equation}
and conversely, by binomial inversion, 
\begin{equation}  \label{N_Z}
\sum_{\substack{n_1,\ldots, n_k\in \N\\ n_1^2+\cdots+ n_k^2\le x}}
f((n_1,\ldots, n_k)) = \frac1{2^k} \sum_{j=0}^{k-1} (-1)^j \binom{k}{j} \sum_{\substack{n_1,\ldots, n_{k-j}\in \Z \\ n_1^2+\cdots+ n_{k-j}^2 \le x}}
f((n_1,\ldots, n_{k-j})),
\end{equation}
showing that estimates for GCD sums over natural numbers imply estimates for the corresponding sums over 
the integers and vice versa.    

\subsection{Some lemmas}

The identity included in the next lemma was already mentioned in the Introduction. We remark that there is no known similar formula
concerning the LCM of the integers $n_1,\ldots,n_k$.

\begin{lemma} \label{Lemma_gcd} Let $f:\N \cup \{0\}\to \C$ be an arbitrary function with $f(0)=0$. Then for every $n\in \N$,
\begin{equation} \label{id_gcd_sum}
\sum_{\substack{n_1,\ldots,n_k\in \Z\\ n_1^2+\cdots +n_k^2=n}}  f((n_1,\ldots,n_k)) = 
\sum_{d^2\mid n} (\mu*f)(d) r_k(n/d^2).
\end{equation}
\end{lemma}

We will apply Euler’s summation formula, formulated as the next lemma. See, e.g., \cite[Th.\ 3.1]{Apo1976}. 

\begin{lemma} \label{Lemma_Euler_sum}
If the function $\psi$ has a continuous derivative $\psi'$
on the interval $[y, x]$, where $0 < y < x$ are real numbers, then
\begin{equation*}
\sum_{y<n\le x} \psi(n) = \int_y^x \psi(t)\, dt + \int_y^x \{t\} \psi'(t)\, dt -\{x\} \psi(x) + \{y\} \psi(y),  
\end{equation*}
$\{x\}=x-\lfloor x \rfloor$ denoting the fractional part of $x$.
\end{lemma}

We need the following results concerning the lattice points in $k$-dimensional ellipsoids. 

\begin{lemma} \label{Lemma_number_k} Let $k\ge 1$ be fixed and let $x, a_1,\ldots,a_k>0$ be real numbers. Then
\begin{equation} \label{number_estimate_k}
\sum_{\substack{n_1,\ldots,n_k\in \N \\ a_1n_1^2+\cdots +a_kn_k^2 \le x}} 1 
= \frac{V_k x^{k/2}}{2^k \sqrt{a_1\cdots a_k}} + O\left(\frac{x^{(k-1)/2}(\sqrt{a_1}+\cdots+\sqrt{a_k})}{\sqrt{a_1\cdots a_k}}\right),
\end{equation}
uniformly in $x, a_1,\ldots,a_k$, where $V_1=1$ and $V_k$ \textup{($k\ge 2$)} are given by \eqref{V_k}.
\end{lemma} 

If $a_1= \cdots = a_k=1$, then this gives formula \eqref{number_lattice_points_k} with a weaker error term.
However, our proof is simple and elementary. 

\begin{lemma} \label{Lemma_prod_k} Let $k\ge 1$ be fixed and let $x,a_1,\ldots,a_k>0$ be real numbers. 
Then
\begin{equation} \label{product_estimate_k}
\sum_{\substack{n_1,\ldots,n_k\in \N \\ a_1n_1^2+\cdots +a_kn_k^2 \le x}} 
n_1\cdots n_k = \frac{x^k}{2^k k! a_1\cdots a_k} + O\left(\frac{x^{k-1/2}(\sqrt{a_1}+\cdots+\sqrt{a_k})}{a_1\cdots a_k}\right),
\end{equation}
uniformly in $x, a_1,\ldots,a_k$.
\end{lemma}

The following lemma is a part of \cite[Lemma\ 3.1]{HilTot2016}, concerning functions
in the class ${\cal A}_r$ to be defined in Section \ref{Sect_Estimates_LCM}, proved
by the Euler product representation of the involved multiple Dirichlet series. 

\begin{lemma} \label{Lemma_HilTot} If $k \ge 2$ and $f\in {\cal A}_r$, where $r\ge 0$ is a real number, then
\begin{equation*}
\sum_{n_1,\ldots,n_k=1}^{\infty}
\frac{f([n_1,\ldots,n_k])}{n_1^{z_1}\cdots n_k^{z_k}}
= \zeta(z_1-r)\cdots \zeta(z_k-r)
H_{f,k}(z_1,\ldots,z_k),
\end{equation*}
where the multiple Dirichlet series 
\begin{equation*}
H_{f,k}(z_1,\ldots,z_k):= \sum_{n_1,\ldots,n_k=1}^{\infty}
\frac{h_{f,k}(n_1,\ldots,n_k)}{n_1^{z_1}\cdots n_k^{z_k}}
\end{equation*}
is absolutely convergent for $z_j\in \C$ with $\Re z_j  > r+1/2$ \textup{($1\le j \le k$)}.
\end{lemma} 

\section{Main results} \label{Sect_Main_results}

\subsection{Spherical summations of arbitrary functions}

For functions $F,G:\N^k\to \C$ ($k\ge 1$) consider their convolution $F*G$ defined by
\begin{equation} \label{convo_k}
(F*G)(n_1,\ldots,n_k)=\sum_{d_1\mid n_1, \ldots, d_k\mid n_k} F(d_1,\ldots,d_k) G(n_1/d_1,\ldots, n_k/d_k), 
\end{equation}
and the generalized M\"obius function $\mu(n_1,\ldots,n_k) = \mu(n_1)\cdots \mu(n_k)$, which is the inverse of the $k$-variable 
constant $1$ function under convolution \eqref{convo_k}. See the survey \cite{Tot2014} on properties of (multiplicative) arithmetic 
functions of several variables. 

Our first result is the following.
\begin{theorem} \label{Th_Wintner_gen} Let $F:\N^k \to \C$ be an arbitrary arithmetic function of $k$ variables 
\textup{($k\ge 1$)} and assume that the multiple 
Dirichlet series
\begin{equation*}
\sum_{n_1,\ldots,n_k=1}^{\infty} \frac{(\mu*F)(n_1,\ldots,n_k)}{n_1^{z_1}\cdots n_k^{z_k}}
\end{equation*}
is absolutely convergent provided that $z_j\in \C$ with $\Re z_j\ge t$ \textup{($1\le j\le k$)}, where $0<t\le 1$ is a real number.

i) If $t=1$,  then 
\begin{equation*}
\lim_{x\to \infty} \frac1{x^{k/2}} \sum_{\substack{n_1,\ldots,n_k\in \N\\ n_1^2+\cdots +n_k^2\le x}} F(n_1,\ldots,n_k) = 
\frac{V_k}{2^k} B_{F,k},
\end{equation*}
where $V_k$ is given by \eqref{V_k}, and 
\begin{equation*} 
B_{F,k}: =\sum_{n_1,\ldots,n_k=1}^{\infty} \frac{(\mu*F)(n_1,\ldots,n_k)}{n_1\cdots n_k}.  
\end{equation*}

ii) If $0<t<1$, then 
\begin{equation*}
\sum_{\substack{n_1,\ldots,n_k\in \N\\ n_1^2+\cdots +n_k^2\le x}} F(n_1,\ldots,n_k) = \frac{V_k}{2^k} B_{F,k} x^{k/2} + O(x^{(k-1+t)/2}).    
\end{equation*}
\end{theorem}

For $k=1$ Part i) of Theorem \ref{Th_Wintner_gen} is Wintner's mean value theorem.  See, e.g., \cite[Th.\ 2.19]{Hil}, \cite[p.\ 138]{Pos1988}.
Also, Part i) is the analog of the corresponding result for summation of functions $F(n_1,\ldots,n_k)$ with $n_1,\ldots,n_k\le x$, 
obtained by Ushiroya \cite{Ush2012}. Note that if $F$ is multiplicative, then
\begin{equation*}
B_{F,k} = \prod_p \left(1-\frac1{p}\right)^k \sum_{\nu_1,\ldots,\nu_k=0}^{\infty} \frac{F(p^{\nu_1},\ldots,p^{\nu_k})}{p^{\nu_1+\cdots +\nu_k}}.    
\end{equation*}

To give an application of Theorem \ref{Th_Wintner_gen} we remark that if $f:\N \to \C$ is an arbitrary arithmetic function, then
\begin{equation*}
\sum_{n_1,\ldots,n_k=1}^{\infty} \frac{f((n_1,\ldots,n_k))}{n_1^{z_1}\cdots n_k^{z_k}} = 
\frac{\zeta(z_1)\cdots \zeta(z_k)}{\zeta(z_1+\cdots + z_k)} \sum_{n=1}^{\infty} \frac{f(n)}{n^{z_1+\cdots+z_k}}, 
\end{equation*}
see \cite[Eq.\, (16)]{Tot2014}. In particular, for $f(n)=\tau(n)$,
\begin{equation*}
\sum_{n_1,\ldots,n_k=1}^{\infty} \frac{\tau((n_1,\ldots,n_k))}{n_1^{z_1}\cdots n_k^{z_k}} = 
\zeta(z_1)\cdots \zeta(z_k) \zeta(z_1+\cdots + z_k),
\end{equation*}
with $z_j\in \C$, $\Re z_j>1$ ($1\le j\le k$).

This shows that for the function $F(n_1,\ldots,n_k)=\tau((n_1,\ldots,n_k))$ we have
\begin{equation} \label{series_tau}
\sum_{n_1,\ldots,n_k=1}^{\infty} \frac{(\mu*F)(n_1,\ldots,n_k)}{n_1^{z_1}\cdots n_k^{z_k}} = 
\zeta(z_1+\cdots + z_k),
\end{equation}
absolutely convergent for $z_1=\cdots =z_k=1$ ($k\ge 2$).

We deduce that for $k\ge 2$,
\begin{equation*}
\lim_{x\to \infty} \frac1{x^{k/2}} \sum_{\substack{n_1,\ldots,n_k\in \N\\ n_1^2+\cdots +n_k^2\le x}} \tau((n_1,\ldots,n_k)) = 
\frac{V_k\zeta(k)}{2^k}
\end{equation*}
and by \eqref{Z_N},
\begin{equation*}
\lim_{x\to \infty} \frac1{x^{k/2}} \sum_{\substack{n_1,\ldots,n_k\in \Z\\ n_1^2+\cdots +n_k^2\le x}} \tau((n_1,\ldots,n_k)) = 
V_k\zeta(k).
\end{equation*}

Moreover, series \eqref{series_tau} is absolutely convergent if $\Re z_j>1/k$ ($1\le j \le k$), hence by Part ii) of Theorem \ref{Th_Wintner_gen}
we deduce the formula with error term
\begin{equation} \label{tau_first}
\sum_{\substack{n_1,\ldots,n_k\in \N\\ n_1^2+\cdots +n_k^2\le x}} \tau((n_1,\ldots,n_k)) = 
\frac{V_k\zeta(k)}{2^k} x^{k/2} + O(x^{(k-1+1/k)/2 + \varepsilon}),
\end{equation}
but this error can be improved, see Corollary \ref{Cor_tau}.

\subsection{Estimates for functions of the GCD} \label{Sect_GCD}

We prove the following estimates.

\begin{theorem} \label{Th_main_1} Let $k\ge 2$ and let $f=g*\1$, where $g$ is a bounded function. Then
\begin{equation*}
\sum_{\substack{n_1,\ldots,n_k\in \Z\\ n_1^2+\cdots +n_k^2\le x}} f((n_1,\ldots,n_k)) = V_k D(g,k)  x^{k/2} + R_k(x),
\end{equation*}
where 
$D(g,k) = \sum_{n=1}^{\infty} \frac{g(n)}{n^k}$,
$R_2(x)\ll \sqrt{x}$, $R_3(x)\ll x^{517/1648+\varepsilon}$, $R_4(x)\ll x (\log x)^{2/3}$, and $R_k(x)\ll x^{k/2-1}$ if $k\ge 5$.
\end{theorem}

This applies, i.e., for the functions $f(n)=\tau(n)$ and $f(n)= \sum_{d\mid n} \mu^2(d)$, representing the number of squarefree divisors of $n$,
giving better error terms than in \eqref{tau_first}. If $f(n)=\tau(n)$, then we deduce the next results.

\begin{cor} \label{Cor_tau} For every $k\ge 2$, 
\begin{equation*}
\sum_{\substack{n_1,\ldots,n_k\in \Z \\ n_1^2+\cdots +n_k^2\le x}} \tau((n_1,\ldots,n_k)) = V_k \zeta(k) x^{k/2} + 
R_k(x),
\end{equation*} 
furthermore, for summation over natural numbers,
\begin{equation*}
\sum_{\substack{n_1,n_2\in \N \\ n_1^2+n_2^2\le x}} \tau((n_1,n_2)) = \frac{\pi^3}{24} x -\frac1{2} \sqrt{x} \log x + 
O(\sqrt{x}),
\end{equation*}
and for $k\ge 3$,
\begin{equation*}
\sum_{\substack{n_1,\ldots,n_k\in \N \\ n_1^2+\cdots +n_k^2\le x}} \tau((n_1,\ldots,n_k)) = \frac1{2^k} \left(V_k\zeta(k) x^{k/2} - k V_{k-1}\zeta(k-1)x^{(k-1)/2} 
\right) + R_k(x),
\end{equation*}
where $V_k$ is defined by \eqref{V_k}, and  $R_k(x)$ is given in Theorem \ref{Th_main_1}. 
\end{cor}

\begin{theorem} \label{Th_main_id} Let $f=g*\id$, where $g$ is a bounded function. Then for $k=2$,
\begin{equation} \label{form_2}
\sum_{\substack{n_1,n_2\in \Z\\ n_1^2+n_2^2\le x}} f((n_1,n_2)) = \frac3{\pi} D(g,2)  x\log x + O(x),
\end{equation}
and for $k\ge 3$,
\begin{equation} \label{form_k_3}
\sum_{\substack{n_1,\ldots,n_k\in \Z\\ n_1^2+\cdots +n_k^2\le x}} f((n_1,\ldots,n_k)) = V_k \frac{\zeta(k-1)}{\zeta(k)} D(g,k) x^{k/2} + Q_k(x),
\end{equation}
where $D(g,k) = \sum_{n=1}^{\infty} \frac{g(n)}{n^k}$, $Q_3(x)\ll x$, $Q_4(x)\ll x (\log x)^{5/3}$, and $Q_k(x)\ll x^{k/2-1}$ if $k\ge 5$.
\end{theorem}

\begin{cor} \label{Cor_g_id} Let $f=g*\id$, where $g$ is a bounded function. Then 
\begin{equation*}
\sum_{\substack{n_1,n_2\in \N \\ n_1^2+n_2^2\le x}} f((n_1,n_2)) = \frac{3}{4\pi} D(g,2) x \log x  + O(x),
\end{equation*}
\begin{equation*}
\sum_{\substack{n_1,n_2,n_3\in \N \\ n_1^2+n_2^2+n_3^2\le x}} f((n_1,n_2,n_3)) = \frac{\pi^3}{36\zeta(3)} D(g,3) x^{3/2} - 
\frac{9}{8\pi}D(g,2)x\log x + O(x),
\end{equation*}
and for $k\ge 4$,
\begin{equation*}
\sum_{\substack{n_1,\ldots,n_k\in \N \\ n_1^2+\cdots +n_k^2\le x}} f((n_1,\ldots,n_k)) 
\end{equation*}
\begin{equation*}
= \frac1{2^k} \left(V_k\frac{\zeta(k-1)}{\zeta(k)} D(g,k) x^{k/2} - k 
V_{k-1}\frac{\zeta(k-2)}{\zeta(k-1)}D(g,k-1) x^{(k-1)/2} 
\right)  + Q_k(x), 
\end{equation*}
where $Q_k(x)$ is given in Theorem \ref{Th_main_id}.
\end{cor}

These results apply, e.g., for the functions $\id=\delta * \id$, $\sigma=\1*\id$, 
$\beta=\lambda*\id$ (alternating sum-of-divisors function, cf. \cite{Tot2013}), 
$\varphi =\mu*\id$ (Euler function), $\psi=\mu^2 * \id$ (Dedekind function).

Next we consider the function $f_{S,\eta}$ implicitly defined by
\begin{equation} \label{f_S_eta_cond}
h_{S,\eta}(n):= (\mu*f_{S,\eta})(n)= \begin{cases} (\log p)^{\eta}, & \text{ if $n=p^\nu$ a prime power with $\nu \in S$},\\ 0, & \text{ otherwise},
\end{cases}
\end{equation}
where $1\in S\subseteq \N$ and $\eta \ge 0$ is real. By M\"{o}bius inversion we obtain that for $n=\prod_p p^{\nu_p(n)}\in \N$,
\begin{equation*}
f_{S,\eta}(n)= \sum_{d\mid n} h_{S,\eta}(d)= \sum_{p\mid n} (\log p)^{\eta} \# \{\nu: 1\le \nu\le \nu_p(n), \nu \in S \},
\end{equation*}
where $f_{S,\eta}(1)=0$ (empty sum). This function was introduced by the authors \cite{HeyTot2021}. 

If $S=\N$, then 
\begin{equation*}
f_{\N,\eta}(n):=  \sum_{p\mid n} \nu_p(n) (\log p)^{\eta},
\end{equation*}
which gives for $\eta=1$, $f_{\N,1}(n)=\log n$, while $h_{\N,1}(n)=\Lambda(n)$ is the von Mangoldt function. For $\eta=0$ one has $f_{\N,0}(n) =\Omega(n)$.

Let $S=\{1\}$. Then
\begin{equation*}
f_{\{1\},\eta}(n):=  \sum_{p\mid n} (\log p)^{\eta},
\end{equation*}
and if $\eta=0$, then $f_{\{1\},0}(n) =\omega(n)$. If $\eta=1$, then $f_{\{1\},1}(n) =\log \kappa(n)$, where $\kappa(n)=\prod_{p\mid n} p$.
Note that $\sum_{n\le x} h_{\{1\},1}(n)=\sum_{p \le x} \log p = \theta(x)$ is the Chebyshev theta function.

\begin{theorem} \label{Th_omega_type} If $1\in S\subseteq \N$, $\eta \ge 0$ is real and $k\ge 2$, then 
\begin{equation*}
\sum_{\substack{n_1,\ldots,n_k\in \Z\\ n_1^2+\cdots +n_k^2\le x}} f_{S,\eta}((n_1,\ldots,n_k)) = V_k 
\Big(\sum_p (\log p)^{\eta} H_{S,k}(p)\Big) x^{k/2} + T_k(x),
\end{equation*}
where
$H_{S,k}:= \sum_{\nu \in S} 1/p^{\nu k}$, $T_k(x)\ll \sqrt{x} (\log x)^{\eta-1}$ for $k\in \{2,3\}$, $T_4(x) \ll x(\log x)^{2/3}$ for $k=4$, and 
$T_k(x)\ll x^{k/2-1}$ for $k\ge 5$.
\end{theorem}

In particular, we deduce
\begin{cor} Let $k\ge 2$, and let $f$ be one of the functions $\log n$, $\log \kappa(n)$, $\omega(n)$, $\Omega(n)$. Then
\begin{equation*} 
\sum_{\substack{n_1,\ldots,n_k\in \Z\\ n_1^2+\cdots +n_k^2\le x}} 
f((n_1,\ldots,n_k)) = K_{f,k} V_k x^{k/2} + T_k(x),
\end{equation*}
where
\begin{equation*} 
K_{\log,k} = \sum_p \frac{\log p }{p^k-1}, \quad K_{\log \kappa,k}= \sum_p \frac{\log p }{p^k}, \quad
K_{\omega,k}= \sum_p \frac1{p^k}, \quad K_{\Omega,k}= \sum_p \frac1{p^k-1},
\end{equation*}
and $T_k(x)$ is given in Theorem \ref{Th_omega_type}.
\end{cor}

\subsection{Estimates for functions of the LCM} \label{Sect_Estimates_LCM}

Given  a fixed positive real number $r$ let ${\cal A}_r$ denote the
class of multiplicative arithmetic functions $f:\N \to \C$
satisfying the following properties: there exist real constants
$C_1,C_2$ such that $|f(p)-p^r|\le C_1\, p^{r-1/2}$ for every prime $p$, and 
$|f(p^{\nu})|\le C_2\, p^{\nu r}$ for every prime power $p^\nu$ with $\nu \ge 2$.
This class of functions was defined by Hilberdink and T\'oth \cite{HilTot2016}. 

Observe that the functions $\id, \sigma, \varphi, \psi$ belong to the class ${\cal A}_1$.
Also, any bounded multiplicative function $f$ such that $f(p)=1$ for every prime $p$, in 
particular $f(n)=\mu^2(n)$ is in the class ${\cal A}_0$.

We prove the following results.

\begin{theorem} \label{Th_1_LCM} Let $k\ge 2$ be a fixed integer and let $f\in {\cal A}_1$ be a function. 
Then 
\begin{equation*} 
\sum_{\substack{n_1,\ldots,n_k\in \N\\ n_1^2+\cdots +n_k^2\le x}} 
f([n_1,\ldots,n_k]) = \frac{C_{f,k}}{2^k k!} x^k + O\left(x^{k-1/4 +\varepsilon}\right),
\end{equation*}
where 
\begin{equation*}
C_{f,k}= \prod_p \left(1-\frac1{p}\right)^k \sum_{\nu_1,\ldots,
\nu_k=0}^{\infty}
\frac{f(p^{\max(\nu_1,\ldots,\nu_k)})}{p^{2(\nu_1 +\cdots +\nu_k)}}.
\end{equation*}
\end{theorem} 

\begin{cor} \label{Cor_1_LCM} Let $k\ge 2$. Then
\begin{equation*}  
\sum_{\substack{n_1,\ldots,n_k\in \N\\ n_1^2+\cdots +n_k^2 \le x}} 
[n_1,\ldots,n_k] = \frac{C_k}{2^k k!} x^k + R_k(x),
\end{equation*}
\begin{equation*} 
\sum_{\substack{n_1,\ldots,n_k\in \Z\\ n_1^2+\cdots +n_k^2 \le x}} 
[n_1,\ldots,n_k] = \frac{C_k}{k!} x^k + R_k(x),
\end{equation*}
where $R_k(x) \ll x^{k-1/4 +\varepsilon}$ for $k\ge 3$, $R_2(x)\ll x^{3/2}\log x$, and  
\begin{equation*} 
C_k= \prod_p \left(1-\frac1{p}\right)^k \sum_{\nu_1,\ldots, \nu_k=0}^{\infty}
\frac1{p^{2(\nu_1 +\cdots +\nu_k)- \max(\nu_1,\ldots,\nu_k)}},
\end{equation*}
in particular, $C_2= \zeta(3)/\zeta(2)$, $C_3 =\zeta(3)\zeta(5) \prod_p \left(1-\frac{3}{p^2}
+\frac{4}{p^3} -\frac{3}{p^4}+\frac1{p^6}\right)$.
\end{cor} 

\begin{theorem} \label{Th_1_LCM__r_zero} Let $k\ge 2$ be a fixed integer and let $f\in {\cal A}_0$ be a function. 
Then 
\begin{equation*} 
\sum_{\substack{n_1,\ldots,n_k\in \N\\ n_1^2+\cdots +n_k^2\le x}} 
f([n_1,\ldots,n_k]) = \frac{V_k E_{f,k}}{2^k} x^{k/2} + O\left(x^{k/2-1/4 +\varepsilon}\right),
\end{equation*}
where $V_k$ is given by \eqref{V_k} and 
\begin{equation*} 
E_{f,k} = \prod_p \left(1-\frac1{p}\right)^k \sum_{\nu_1,\ldots, \nu_k=0}^{\infty}
\frac{f(p^{\max(\nu_1,\ldots,\nu_k)})}{p^{\nu_1 +\cdots +\nu_k}}.
\end{equation*}
\end{theorem} 

\begin{cor} Let $k\ge 2$. Then 
\begin{equation*} 
\sum_{\substack{n_1,\ldots,n_k\in \N\\ n_1^2+\cdots +n_k^2\le x}} 
\mu^2([n_1,\ldots,n_k]) = \frac{V_k}{(2\zeta(2))^k} x^{k/2} + O\left(x^{k/2-1/4 +\varepsilon}\right).
\end{equation*}
\end{cor} 

\section{Proofs}

\subsection{Proofs of the lemmas}

\begin{proof}[Proof of Lemma {\rm \ref{Lemma_gcd}}] We have by the convolutional identity $f= \1*(\mu *f)$,
\begin{equation*}
\sum_{\substack{n_1,\ldots,n_k\in \Z\\ n_1^2+\cdots +n_k^2=n}} f((n_1,\ldots,n_k)) =  
\sum_{\substack{n_1,\ldots,n_k\in \Z\\ n_1^2+\cdots +n_k^2=n}} \sum_{d\mid (n_1,\ldots,n_k)} (\mu*f)(d) 
= \sum_{\substack{a_1,\ldots,a_k\in \Z\\ d^2(a_1^2+\cdots +a_k^2)=n}} (\mu*f)(d)
\end{equation*}
\begin{equation*}
 = \sum_{d^2\mid n} (\mu*f)(d) \sum_{\substack{a_1,\ldots,a_k\in \Z\\ a_1^2+\cdots +a_k^2=n/d^2}} 1 = \sum_{d^2\mid n} (\mu*f)(d) r_k(n/d^2).
\end{equation*}
\end{proof}

\begin{proof}[Proof of Lemma {\rm \ref{Lemma_number_k}}] By induction on $k$. Estimate \eqref{number_estimate_k} is true for $k=1$:  
\begin{equation*}
\sum_{\substack{n_1\in \N \\ a_1 n_1^2 \le x}} 1 = \sum_{1\le\ n_1 \le \sqrt{x/a_1}} 1 =
\sqrt{\frac{x}{a_1}}+  O(1).
\end{equation*}

Assume that \eqref{number_estimate_k} is true for $k-1$, where $k\ge 2$. Then
\begin{equation*} 
\sum_{\substack{n_1,\ldots,n_k\in \N \\ a_1n_1^2+\cdots +a_kn_k^2 \le x}} 
1 = \sum_{1\le n_k\le \sqrt{x/a_k}} \sum_{\substack{n_1,\ldots,n_{k-1}\in \N\\ a_1n_1^2+\cdots +a_{k-1}n_{k-1}^2 \le x-a_kn_k^2}} 1 
\end{equation*}
\begin{equation*} 
= \sum_{1\le n_k\le \sqrt{x/a_k}}  \left(\frac{V_{k-1}(x-a_kn_k^2)^{(k-1)/2}}{2^{k-1} \sqrt{a_1\cdots a_{k-1}}} + O\left(\frac{(x-a_kn_k^2)^{(k-2)/2}(\sqrt{a_1}+\cdots+\sqrt{a_{k-1}})}{\sqrt{a_1\cdots a_{k-1}}} \right) \right)
\end{equation*}
\begin{equation} \label{first_term}
= \frac{V_{k-1}a_k^{(k-1)/2}}{2^{k-1} \sqrt{a_1\cdots a_{k-1}}} \sum_{1\le n_k\le \sqrt{x/a_k}} (x/a_k - n_k^2)^{(k-1)/2} 
\end{equation}
\begin{equation*}
+ O\left(\frac{x^{(k-2)/2}(\sqrt{a_1}+\cdots+\sqrt{a_{k-1}})}{\sqrt{a_1\cdots a_{k-1}}} \sum_{1\le n_k\le \sqrt{x/a_k}} 
1 \right).
\end{equation*}

Here the error term is 
\begin{equation*} 
\ll \frac{x^{(k-2)/2}(\sqrt{a_1}+\cdots+\sqrt{a_{k-1}})}{\sqrt{a_1\cdots a_{k-1}}} \Big(\frac{x}{a_k}\Big)^{1/2} = 
\frac{x^{(k-1)/2}(\sqrt{a_1}+\cdots+\sqrt{a_{k-1}})}{\sqrt{a_1\cdots a_{k-1}a_k}}.
\end{equation*}

To estimate the main term we use Euler's summation formula from Lemma \ref{Lemma_Euler_sum}.
By choosing the function $\psi(t)= (x^2-t^2)^{(k-1)/2}$ we deduce for $k\ge 2$ that
\begin{equation*}
\sum_{1\le n \le x} (x^2-n^2)^{(k-1)/2} = \int_0^x (x^2-t^2)^{(k-1)/2}\, dt + O(x^{k-1}),     
\end{equation*}
where 
\begin{equation*}
\int_0^x (x^2-t^2)^{(k-1)/2}\, dt = x^{k} I_k
\end{equation*}
with
\begin{equation*}
I_k:=  \int_0^{\pi/2} (\cos t)^k \, dt = 
\begin{cases}
\frac{(2m-1)!!}{(2m)!!}\cdot \frac{\pi}{2}, & \text {if $k=2m$},\\
\frac{(2m)!!}{(2m+1)!!}, & \text {if $k=2m+1$}.
\end{cases}
\end{equation*}

Hence the main term in \eqref{first_term} is 
\begin{equation*} 
\frac{V_{k-1}a_k^{(k-1)/2}}{2^{k-1} \sqrt{a_1\cdots a_{k-1}}} \left(I_k (x/a_k)^{k/2} + O((x/a_k)^{(k-1)/2}) \right) 
\end{equation*}
\begin{equation*} 
= \frac{V_{k-1}I_k x^{k/2}}{2^{k-1} \sqrt{a_1\cdots a_{k-1}a_k}} + O\left(\frac{x^{(k-1)/2}}{\sqrt{a_1\cdots a_{k-1}}} \right). 
\end{equation*}

Here $V_{k-1} I_k= V_k/2$, and the proof is complete.
\end{proof}

\begin{proof}[Proof of Lemma {\rm \ref{Lemma_prod_k}}] By induction on $k$, similar to the proof of Lemma \ref{Lemma_number_k}, but here we only
need the familiar formula
\begin{equation*}
\sum_{1\le n\le x} n^k = \frac{x^{k+1}}{k+1} +O(x^k) \quad (k\in \N).
\end{equation*}

Estimate \eqref{product_estimate_k} is true for $k=1$:  
\begin{equation*}
\sum_{\substack{n_1\in \N \\ a_1 n_1^2 \le x}}  n_1 = \sum_{1\le\ n_1 \le \sqrt{x/a_1}} n_1 =
\frac{x}{2a_1}+ O\left(\sqrt{\frac{x}{a_1}}\, \right).
\end{equation*}

Assume that \eqref{product_estimate_k} is true for $k-1$, where $k\ge 2$. Then
\begin{equation*} 
\sum_{\substack{n_1,\ldots,n_k\in \N \\ a_1n_1^2+\cdots +a_kn_k^2 \le x}} 
n_1\cdots n_k = \sum_{1\le n_k\le \sqrt{x/a_k}} n_k
\sum_{\substack{n_1,\ldots,n_{k-1}\in \N\\ a_1n_1^2+\cdots +a_{k-1}n_{k-1}^2 \le x-a_kn_k^2}} 
n_1\cdots n_{k-1} 
\end{equation*}
\begin{equation*} 
= \sum_{1\le n_k\le \sqrt{x/a_k}} n_k  \left( \frac{(x-a_kn_k^2)^{k-1}}{2^{k-1} (k-1)! a_1\cdots a_{k-1}} + O\left(\frac{(x-a_kn_k^2)^{k-3/2}(\sqrt{a_1}+\cdots+\sqrt{a_{k-1}})}{a_1\cdots a_{k-1}} \right) \right)
\end{equation*}
\begin{equation} \label{first_term_1}
= \frac1{2^{k-1} (k-1)! a_1\cdots a_{k-1}} \sum_{1\le n_k\le \sqrt{x/a_k}} n_k (x-a_kn_k^2)^{k-1} 
\end{equation}
\begin{equation*}
+ O\left(\frac{x^{k-3/2}(\sqrt{a_1}+\cdots+\sqrt{a_{k-1}})}{a_1\cdots a_{k-1}} \sum_{1\le n_k\le \sqrt{x/a_k}} 
n_k \right).
\end{equation*}

Here the error term is 
\begin{equation*} 
\ll \frac{x^{k-3/2}(\sqrt{a_1}+\cdots+\sqrt{a_{k-1}})}{a_1\cdots a_{k-1}} \frac{x}{a_k} = 
\frac{x^{k-1/2}(\sqrt{a_1}+\cdots+\sqrt{a_{k-1}})}{a_1\cdots a_{k-1}a_k}, 
\end{equation*}
and for the sum $S$ in the main term we have
\begin{equation*} 
S:= \sum_{1\le n_k\le \sqrt{x/a_k}} n_k (x-a_kn_k^2)^{k-1} = \sum_{1\le n_k\le \sqrt{x/a_k}} n_k \sum_{j=0}^{k-1} (-1)^j 
\binom{k-1}{j} x^{k-1-j} a_k^j n_k^{2j}
\end{equation*}
\begin{equation*} 
= \sum_{j=0}^{k-1} (-1)^j \binom{k-1}{j} x^{k-1-j} a_k^j \sum_{1\le n_k\le \sqrt{x/a_k}} n_k^{2j+1}
\end{equation*}
\begin{equation*} 
= \sum_{j=0}^{k-1} (-1)^j \binom{k-1}{j} x^{k-1-j} a_k^j \left( \frac{(x/a_k)^{j+1}}{2j+2} +O((x/a_k)^{j+1/2})
\right)
\end{equation*}
\begin{equation*} 
= \frac{x^k}{2a_k} A_k + O\left(x^{k-1/2}/\sqrt{a_k}\right),
\end{equation*}
where
\begin{equation*} 
A_k: = \sum_{j=0}^{k-1} \frac{(-1)^j}{j+1} \binom{k-1}{j}= \frac1{k}.
\end{equation*}

This gives that the term in \eqref{first_term_1} is
\begin{equation*} 
\frac{x^k}{2^k k! a_1\cdots a_k} + O\left(\frac{x^{k-1/2} \sqrt{a_k}}{a_1\cdots a_k}\right),
\end{equation*}
showing that \eqref{product_estimate_k} is true for $k$.
\end{proof}

\subsection{Proof of the generalization of Wintner's theorem} 

\begin{proof}[Proof of Theorem {\rm \ref{Th_Wintner_gen}}]

We have by the identity $F=\1* (\mu*F)$, here with functions of $k$ variables, where $\mu(n_1,\ldots,n_k)=\mu(n_1)\cdots \mu(n_k)$,
\begin{equation*}
S_{F,k}(x):=\sum_{\substack{n_1,\ldots,n_k\in \N\\ n_1^2+\cdots +n_k^2\le x}} F(n_1,\ldots,n_k)  
= \sum_{\substack{n_1,\ldots,n_k\in \N\\ n_1^2+\cdots +n_k^2\le x}} \sum_{d_1\mid n_1,\ldots, d_k\mid n_k} (\mu*F)(d_1,\ldots,d_k) 
\end{equation*}
\begin{equation*}
= \sum_{\substack{d_1,a_1\ldots,d_k,a_k\in \N\\ d_1^2a_1^2+\cdots +d_k^2a_k^2\le x}} (\mu*F)(d_1,\ldots,d_k) 
= \sum_{\substack{d_1,\ldots,d_k\in \N\\ d_1,\ldots,d_k\le \sqrt{x}}} (\mu*F)(d_1,\ldots,d_k) 
\sum_{\substack{a_1,\ldots,a_k\in \N\\ d_1^2a_1^2+\cdots +d_k^2a_k^2\le x}} 1.
\end{equation*}

By using Lemma \ref{Lemma_number_k} we deduce 
\begin{equation*}
S_{F,k}(x)=  \sum_{d_1,\ldots,d_k\le \sqrt{x}} (\mu*F)(d_1,\ldots,d_k) \left(
\frac{V_k x^{k/2}}{2^k d_1\cdots d_k} + O\left(\frac{x^{(k-1)/2}(d_1+\cdots+d_k)}{d_1\cdots d_k}\right)\right)
\end{equation*}
\begin{equation*}
= \frac{V_kx^{k/2}}{2^k} \sum_{d_1,\ldots,d_k\le \sqrt{x}} \frac{(\mu*F)(d_1,\ldots,d_k)}{d_1\cdots d_k} + R_{F,k}(x),
\end{equation*}
where
\begin{equation*}
R_{F,k}(x)\ll
x^{(k-1)/2} \sum_{d_1,\ldots,d_k\le \sqrt{x}} \frac{|(\mu*F)(d_1,\ldots,d_k)| (d_1+\cdots+d_k)}{d_1\cdots d_k},
\end{equation*}
and
\begin{equation} \label{estimate_error_j}
\frac{R_{F,k}(x)}{x^{k/2}} \ll \frac1{\sqrt{x}} \sum_{j=1}^k \sum_{d_1,\ldots,d_k\le \sqrt{x}} 
\frac{|(\mu*F)(d_1,\ldots,d_k)| d_j}{d_1\cdots d_k}.
\end{equation}

i) Assume that the series
\begin{equation*}
\sum_{d_1,\ldots,d_k=1}^{\infty} \frac{|(\mu*F)(d_1,\ldots,d_k)|}{d_1\cdots d_k}
\end{equation*}
is convergent. Then for a small $\varepsilon>0$ split the inner sum in \eqref{estimate_error_j} for $j=1$ (and similarly for every $j$) in two parts:
\begin{equation*} 
\frac1{\sqrt{x}} \sum_{d_1,\ldots,d_k\le \sqrt{x}}  \frac{|(\mu*F)(d_1,\ldots,d_k)| d_1}{d_1\cdots d_k}
\end{equation*}
\begin{equation*} 
= \frac1{\sqrt{x}} \sum_{\substack{d_1\le \varepsilon \sqrt{x}\\ d_2,\ldots,d_k\le \sqrt{x}}}  \frac{|(\mu*F)(d_1,\ldots,d_k)| d_1}{d_1\cdots d_k}
+ \frac1{\sqrt{x}} \sum_{\substack{\varepsilon \sqrt{x}<d_1\le \sqrt{x} \\ d_2,\ldots,d_k\le \sqrt{x}}}  \frac{|(\mu*F)(d_1,\ldots,d_k)| d_1}{d_1\cdots d_k}
\end{equation*}
\begin{equation*} 
\le \varepsilon \sum_{d_1,\ldots,d_k=1}^{\infty} \frac{|(\mu*F)(d_1,\ldots,d_k)|}{d_1\cdots d_k}
+ \sum_{\substack{\varepsilon \sqrt{x}< d_1 \\ d_2,\ldots,d_k\ge 1}}  \frac{|(\mu*F)(d_1,\ldots,d_k)|}{d_1\cdots d_k}.
\end{equation*}

Here the first term is arbitrary small if $\varepsilon$ is small, and the second term is also arbitrary small if $x$ is large enough (by the definition of multiple convergent series). Hence $R_{F,k}(x)/x^{k/2} \to 0$, as $x\to \infty$. This proves the first part of the theorem.

ii) Assume that $0<t<1$, that is, the series
\begin{equation*}
\sum_{n_1,\ldots,n_k=1}^{\infty} \frac{(\mu*F)(n_1,\ldots,n_k)}{n_1^{z_1}\cdots n_k^{z_k}}
\end{equation*}
is absolutely convergent for $\Re z_j\ge t$ ($1\le j\le k$). Then 
\begin{equation*}
S_{F,k}(x) =  \frac{V_kx^{k/2}}{2^k} \sum_{d_1,\ldots,d_k=1}^{\infty} \frac{(\mu*F)(d_1,\ldots,d_k)}{d_1\cdots d_k}
+O\left( x^{k/2} \sideset{}{'} \sum_{d_1,\ldots,d_k} \frac{|(\mu*F)(d_1,\ldots,d_k)|}{d_1\cdots d_k} \right) 
\end{equation*}
\begin{equation} \label{next_error}
+ O\left( x^{(k-1)/2} \sum_{d_1,\ldots,d_k\le \sqrt{x}} \frac{|(\mu*F)(d_1,\ldots,d_k)|(d_1+\cdots+d_k)}{d_1\cdots d_k}
\right)
\end{equation}
where $\sum^{'}$ means that $d_1,\ldots,d_k\le \sqrt{x}$ does not hold. That is, there exists at least one $m$ ($1\le m\leq k$) 
such that $d_m>\sqrt{x}$. Without loss of generality, we can suppose that $m=1$.

We obtain that
\begin{equation*}
\sideset{}{'}  \sum_{\substack{d_1,\ldots,d_k\\ d_1>\sqrt{x}}} \frac{|(\mu*F)(d_1,\ldots,d_k)|}{d_1\cdots d_k} 
\le \sideset{}{'} \sum_{\substack{d_1,\ldots,d_k\\ d_1>\sqrt{x}}} \frac{|(\mu*F)(d_1,\ldots,d_k)|}{d_1\cdots d_k} 
\left(\frac{d_1}{\sqrt{x}}\right)^{1-t} 
\end{equation*}
\begin{equation*}
\le x^{(t-1)/2} \sum_{d_1,\ldots,d_k=1}^{\infty} \frac{|(\mu*F)(d_1,\ldots,d_k)|}{d_1^t d_2 
\cdots d_k} \ll x^{(t-1)/2}, 
\end{equation*}
since the latter series converges. This gives the error $O(x^{(k-1+t)/2})$.

For the error term in \eqref{next_error} by taking $d_1$ in the numerator (similarly for $d_2,\ldots,d_k$),
\begin{equation*}
\sum_{d_1,\ldots,d_k\le \sqrt{x}} \frac{|(\mu*F)(d_1,\ldots,d_k)|d_1}{d_1\cdots d_k}
= \sum_{d_1,\ldots,d_k\le \sqrt{x}} \frac{|(\mu*F)(d_1,\ldots,d_k)|d_1^t}{d_1^t d_2\cdots d_k}  
\end{equation*}
\begin{equation*}
\le x^{t/2} \sum_{d_1,\ldots,d_k=1}^{\infty} \frac{|(\mu*F)(d_1,\ldots,d_k)|}{d_1^t d_2\cdots d_k}  \ll x^{t/2},
\end{equation*}
the latter series (the same as above) being convergent. This gives the same error, namely $O(x^{(k-1+t)/2})$, and completes the proof.
\end{proof}

\subsection{Proofs of the results for functions of the GCD}

According to \eqref{id_gcd_sum} and \eqref{number_lattice_points_k}, for every arithmetic function $f$, with $f(0)=0$ one has
\begin{equation*}
S_{f,k}(x):=\sum_{\substack{n_1,\ldots,n_k\in \Z\\ n_1^2+\cdots +n_k^2\le x}} f((n_1,\ldots,n_k)) = \sum_{d^2e\le x} (\mu*f)(d) r_k(e) 
= \sum_{d\le \sqrt{x}} (\mu*f)(d) \sum_{e\le x/d^2} r_k(e) 
\end{equation*}
\begin{equation} \label{S_gen}
= V_kx^{k/2} \sum_{d\le \sqrt{x}} \frac{(\mu*f)(d)}{d^k} + \sum_{d\le \sqrt{x}} (\mu*f)(d) P_k(x/d^2).  
\end{equation}

\begin{proof}[Proof of Theorem {\rm \ref{Th_main_1}}] 

Assume that $f=g*\1$, that is, $f(n)=\sum_{d\mid n} g(d)$ ($n\in \N$), where $g$ is a bounded function with 
$|g(n)|\le K$ ($n\in \N$). Then $\mu * f = \mu* g * \1=g$. Hence $|(\mu*f)(n)|\le K$ for every 
$n\in \N$, and the series
\begin{equation*}
\sum_{n=1}^{\infty} \frac{(\mu*f)(n)}{n^k} = \sum_{n=1}^{\infty} \frac{g(n)}{n^k}     
\end{equation*}
is absolutely convergent for every $k\ge 2$. From \eqref{S_gen} we obtain
\begin{equation*}
 S_{f,k}(x)= V_kx^{k/2} \sum_{d=1}^{\infty} \frac{g(d)}{d^k} + O\left(x^{k/2}\sum_{d>\sqrt{x}} \frac1{d^k} \right) + 
 O\left( \sum_{d\le \sqrt{x}} P_k(x/d^2)\right). 
\end{equation*}

Here the first error term is $O(\sqrt{x})$ and using the known estimates for $P_k(x)$ given in Section 
\ref{Sect_Known_estimates} we obtain the indicated error terms by usual elementary estimates. 
\end{proof}

\begin{proof}[Proof of Corollary {\rm \ref{Cor_tau}}] 
For the second part, namely for summation over the natural numbers, use \eqref{N_Z} and the estimate $\sum_{n\le x} \tau(n) = x\log x+ O(x)$ (this is sufficient).
\end{proof}

\begin{proof}[Proof of Theorem {\rm \ref{Th_main_id}}] 

Now let $f=g*\id$, where $g$ is a bounded function with $|g(n)|\le K$ ($n\in \N$). Then $\mu * f = \mu* g *\id = g*\varphi$ and $|(\mu*f)(n)|\le
\sum_{d\mid n} \varphi(d)|g(n/d)| \le K \sum_{d\mid n} \varphi(d) =K n$ for every $n\in \N$. This shows that for $k\ge 3$ the series
\begin{equation*}
\sum_{n=1}^{\infty} \frac{(\mu*f)(n)}{n^k} = \sum_{n=1}^{\infty} \frac{(g*\varphi)(n)}{n^k} = \frac{\zeta(k-1)}{\zeta(k)} \sum_{n=1}^{\infty} \frac{g(n)}{n^k}   
\end{equation*}
is absolutely convergent, and from \eqref{S_gen} we obtain
\begin{equation*}
 S_{f,k}(x)= V_kx^{k/2} \frac{\zeta(k-1)}{\zeta(k)} \sum_{n=1}^{\infty} \frac{g(n)}{n^k} + O\left(x^{k/2}\sum_{d>\sqrt{x}} \frac1{d^{k-1}} \right) + 
 O\left( \sum_{d\le \sqrt{x}} d P_k(x/d^2)\right). 
\end{equation*}

Here the first error term is $O(x)$ and use the known estimates for $P_k(x)$ given in Section 
\ref{Sect_Known_estimates}. This proves \eqref{form_k_3}.

If $k=2$, then 
\begin{equation*}
S_{f,2}(x)
= V_2 x \sum_{d\le \sqrt{x}} \frac{(g*\varphi)(d)}{d^2} + O\left(\sum_{d\le \sqrt{x}} d P_2(x/d^2)\right).  
\end{equation*}

Using the known estimate
\begin{equation*}
\sum_{n\le x} \frac{\varphi(n)}{n^2} = \frac{6}{\pi^2} \log x  + O(1) 
\end{equation*}
(sufficient here in this form) we deduce
\begin{equation*}
\sum_{n\le x} \frac{(g*\varphi)(n)}{n^2} = \sum_{d\le x} \frac{g(d)}{d^2}\sum_{e\le x/d} \frac{\varphi(e)}{e^2}
=\frac{6}{\pi^2} \left(\sum_{d=1}^\infty \frac{g(d)}{d^2}\right) \log x + O(1),
\end{equation*}
which leads, together with $V_2=\pi$ and $P_2(x)\ll x^{517/1648+\varepsilon}$ to formula \eqref{form_2}.
\end{proof}

\begin{proof}[Proof of Corollary {\rm \ref{Cor_g_id}}] 

Follows by \eqref{N_Z} and the estimate 
\begin{equation*}
\sum_{n\le x} (g*\id)(n)= \frac{x^2}{2} \sum_{n=1}^{\infty} \frac{g(n)}{n^2}+ O(x\log x),
\end{equation*}
valid for every bounded function $g$.
\end{proof}

\begin{proof}[Proof of Theorem {\rm \ref{Th_omega_type}}] 
For the function $f_{S,\eta}$ we have by \eqref{f_S_eta_cond},
\begin{equation*} 
\sum_{d\le \sqrt{x}} \frac{(\mu*f_{S,\eta})(d)}{d^k} = \sum_{\substack{p^\nu \le \sqrt{x}\\ \nu \in S}} \frac{(\log p)^{\eta}}{p^{k\nu}} = \sum_{p\le \sqrt{x}}
(\log p)^{\eta} \sum_{\substack{1\le \nu \le m\\ \nu \in S}} \frac1{p^{k\nu}} 
\end{equation*}
\begin{equation} \label{A}
= \sum_{p\le \sqrt{x}} (\log p)^{\eta} \left( H_{S,k}(p) - \sum_{\substack{ \nu \ge m+1\\ \nu \in S}} \frac1{p^{k\nu}}\right),
\end{equation}
where $m=:\lfloor \frac{\log x}{2\log p}\rfloor$, and for every prime $p$,
\begin{equation} \label{H_S_k}
\frac1{p^k}\le H_{S,k}(p):=  \sum_{\nu \in S} \frac1{p^{k\nu}}\le \sum_{\nu=1}^{\infty} \frac1{p^{k\nu}} = \frac1{p^k-1},
\end{equation}
using that $1\in S$. Here
\begin{equation} \label{A_A}
\sum_{p\le \sqrt{x}} (\log p)^{\eta} H_{S,k}(p) = \sum_p (\log p)^{\eta} H_{S,k}(p) - \sum_{p>\sqrt{x}} (\log p)^{\eta} H_{S,k}(p),
\end{equation}
where the series is absolutely convergent by \eqref{H_S_k}, and the last sum is
\begin{equation*}
\ll  \sum_{p>\sqrt{x}} \frac{(\log p)^{\eta}}{p^k-1} \ll  \sum_{p>\sqrt{x}} \frac{(\log p)^{\eta}}{p^k} \ll \frac{(\log x)^{\eta-1}}{x^{(k-1)/2}},
\end{equation*}
see \cite[Lemma\, 3.3]{HeyTot2021}. Also,
\begin{equation} \label{A_1}
A_1:= \sum_{p\le \sqrt{x}}  (\log p)^{\eta} \sum_{\substack{\nu \ge m+1 \\ \nu \in S}} \frac1{p^{k\nu}}
\ll \sum_{p\le \sqrt{x}} (\log p)^\eta \sum_{\nu \ge m+1} \frac1{p^{k\nu}}
=  \sum_{p\le \sqrt{x}} \frac{(\log p)^\eta}{p^{km}(p^k-1)}.
\end{equation}

By the definition of $m$ we have $m> \frac{\log x}{2\log p}-1$, hence $p^{km}>\frac{x^{k/2}}{p^k}$. Thus the last sum in \eqref{A_1} is
\begin{equation*} 
\le  \frac1{x^{k/2}} \sum_{p\le \sqrt{x}} \frac{p^k(\log p)^\eta}{p^k-1}\ll  \frac1{x^{k/2}} \sum_{p\le \sqrt{x}} (\log p)^\eta \le  \frac1{x^{k/2}}
(\log x)^{\eta} \pi(\sqrt{x}),
\end{equation*}
hence
\begin{equation} \label{A_2}
A_1 \ll \frac{(\log x)^{\eta-1}}{x^{(k-1)/2}},
\end{equation}
using $\eta \ge 0$ and the estimate $\pi(\sqrt{x})\ll \frac{\sqrt{x}}{\log x}$.

We deduce by \eqref{S_gen}, \eqref{A}, \eqref{A_A}, \eqref{A_1} and \eqref{A_2} that
\begin{equation} \label{S_with_last_sum}
S_{f_{S,\eta},k}(x)
= V_k x^{k/2} \sum_p (\log p)^{\eta} H_{S,k}(p) + O(\sqrt{x}(\log x)^{\eta-1}) + \sum_{\substack{p^{\nu} \le \sqrt{x}\\ \nu \in S}} 
(\log p)^{\eta} P_k(x/p^{2\nu}).  
\end{equation}

Here the last sum can be estimated by using the known estimates for $P_k(x)$ given in Section 
\ref{Sect_Known_estimates}. For example, if $k=2$, then $P_2(x)\ll x^{\vartheta}$ with $\vartheta:= 517/1648+\varepsilon$. We deduce that
the last sum in \eqref{S_with_last_sum} is, with the notation $m=:\lfloor \frac{\log x}{2\log p}\rfloor$ of above,
\begin{equation*}
\ll \sum_{p^\nu \le \sqrt{x}} (\log p)^{\eta} \Big(\frac{x}{p^{2\nu}}\Big)^{\vartheta} =x^{\vartheta} \sum_{p\le \sqrt{x}} (\log p)^{\eta}
\sum_{\nu \le m}\frac1{p^{2\nu \vartheta}} \le x^{\vartheta} \sum_{p\le \sqrt{x}} \frac{(\log p)^{\eta}}{p^{2\vartheta}}
\end{equation*}
\begin{equation*}
\ll  x^{\vartheta} \frac{(\log x)^{\eta-1}}{x^{(\vartheta-1)/2}} = \sqrt{x} (\log x)^{\eta-1}, 
\end{equation*}
by using \cite[Lemma\, 3.4]{HeyTot2021}. The cases $k=3$, $k=4$ and $k \ge 5$ are similar, and lead to the stated error terms.
\end{proof}

\subsection{Proofs of the results for functions of the LCM}

\begin{proof}[Proof of Theorem {\rm \ref{Th_1_LCM}}] 

Let $f$ be a function in class ${\cal A}_1$. From Lemma \ref{Lemma_HilTot} with $r=1$ we deduce the convolutional identity
\begin{equation*}
f([n_1,\ldots,n_k]) =\sum_{j_1 d_1=n_1,\ldots,j_k d_k=n_k}
j_1\cdots j_k h_{f,k}(d_1,\ldots,d_k).
\end{equation*}

Therefore
\begin{equation*}
S:= \sum_{\substack{n_1,\ldots,n_k\in \N\\ n_1^2+\cdots +n_k^2 \le x}} 
f([n_1,\ldots,n_k]) = 
\sum_{\substack{j_1,d_1,\ldots,j_k,d_k\in \N\\ j_1^2d_1^2+\cdots +j_k^2d_k^2\le x}} 
j_1\cdots j_k h_{f,k}(d_1,\ldots,d_k)
\end{equation*}
\begin{equation*}
= \sum_{1\le d_1,\ldots,d_k \le \sqrt{x}} h_{f,k}(d_1,\ldots,d_k) 
\sum_{\substack{j_1,\ldots,j_k\in \N\\ d_1^2j_1^2+\cdots +d_k^2j_k^2 \le x}} 
j_1\cdots j_k.
\end{equation*}

By Lemma \ref{Lemma_prod_k} we have
\begin{equation*}
S = \sum_{1\le d_1,\ldots,d_k \le \sqrt{x}} h_{f,k}(d_1,\ldots,d_k) \left( \frac{x^k}{2^k k!d_1^2\cdots d_k^2}+ 
O\left( \frac{x^{k-1/2}(d_1+\cdots +d_k)}{d_1^2\cdots d_k^2} \right)
\right)
\end{equation*}
\begin{equation*}
= \frac{x^k}{2^k k!} \sum_{d_1,\ldots,d_k=1}^{\infty} \frac{h_{f,k}(d_1,\ldots,d_k)}{d_1^2\cdots d_k^2} 
+ \left(x^k \sideset{}{'} \sum_{d_1,\ldots,d_k} \frac{|h_{f,k}(d_1,\ldots,d_k)|}{d_1^2\cdots d_k^2} \right) 
\end{equation*}
\begin{equation} \label{second_error}
+ O\left( x^{k-1/2}\sum_{1\le d_1,\ldots,d_k\le \sqrt{x}} \frac{|h_{f,k}(d_1,\ldots,d_k)|(d_1+\cdots+d_k)}{d_1^2\cdots d_k^2}
\right)
\end{equation}
where $\sum^{'}$ means that $d_1,\ldots,d_k\le \sqrt{x}$ does not hold. That is, there exists at least one $m$ ($1\le m\leq k$) 
such that $d_m>\sqrt{x}$. Without loss of generality, we can suppose that $m=1$. We obtain for $0<\varepsilon <1/4$ that
\begin{equation*}
\sideset{}{'}  \sum_{\substack{d_1,\ldots,d_k\\ d_1>\sqrt{x}}} \frac{|h_{f,k}(d_1,\ldots,d_k)|}{d_1^2\cdots d_k^2} 
\le \sideset{}{'} \sum_{\substack{d_1,\ldots,d_k\\ d_1>\sqrt{x}}} \frac{|h_{f,k}(d_1,\ldots,d_k)|}{d_1^2\cdots d_k^2} \left(\frac{d_1}{\sqrt{x}}\right)^{1/2-2\varepsilon} 
\end{equation*}
\begin{equation*}
\le x^{\varepsilon-1/4} \sum_{d_1,\ldots,d_k=1}^{\infty} \frac{|h_{f,k}(d_1,\ldots,d_k)|}{d_1^{3/2+2\varepsilon}d_2^2 
\cdots d_k^2} \ll x^{\varepsilon-1/4}, 
\end{equation*}
since the latter series converges by Lemma \ref{Lemma_HilTot}.

For the error term in \eqref{second_error} by taking $d_1$ in the numerator (similarly for $d_2,\ldots,d_k$),
\begin{equation*}
\sum_{1\le d_1,\ldots,d_k\le \sqrt{x}} \frac{|h_{f,k}(d_1,\ldots,d_k)|d_1}{d_1^2\cdots d_k^2}
= \sum_{1\le d_1,\ldots,d_k\le \sqrt{x}} \frac{|h_{f,k}(d_1,\ldots,d_k)|d_1^{1/2+2\varepsilon}}{d_1^{3/2+2\varepsilon}
d_2^2 \cdots d_k^2}  
\end{equation*}
\begin{equation*}
\le x^{1/4+\varepsilon} \sum_{d_1,\ldots,d_k=1}^{\infty} \frac{|h_{f,k}(d_1,\ldots,d_k)|}{d_1^{3/2+2\varepsilon} 
d_2^2 \cdots d_k^2}  \ll x^{1/4+\varepsilon},
\end{equation*}
the latter series, the same as above, being convergent. This completes the proof.
\end{proof}

\begin{proof} [Proof of Corollary {\rm \ref{Cor_1_LCM}}] If $k=2$, then Theorem \ref{Th_1_LCM} provides the error $O(x^{7/4+\varepsilon})$. 
We show that for $k=2$ and the function $f(n)=n$, the error term is $O(x^{3/2}\log x)$. To this end 
we remark that 
\begin{equation} \label{series_lcm_2}
D(z_1.z_2):= \sum_{n_1,n_2=1}^{\infty} \frac{[n_1,n_2]}{n_1^{z_1}n_2^{z_2}} = \zeta(z_1-1)\zeta(z_2-1)\frac{\zeta(z_1+z_2-1)}{\zeta(z_1+z_2-2)}. 
\end{equation}

To present a short direct proof of this identity, write $n_1=da_1$, $n_2=da_2$ with $(a_1,a_2)=1$. Then $[n_1,n_2]=da_1a_2$ and we 
deduce
\begin{equation*} 
D(z_1,z_2) = \sum_{\substack{d,a_1,a_2=1\\(a_1,a_2)=1}}^{\infty} 
\frac{da_1a_2}{(da_1)^{z_1}(da_2)^{z_2}} =\sum_{d,a_1,a_2=1}^{\infty} 
\frac{da_1a_2}{(da_1)^{z_1}(da_2)^{z_2}} \sum_{\delta \mid (a_1,a_2)} \mu(\delta)
\end{equation*}
and by denoting $a_1=\delta b_1$, $a_2=\delta b_2$,
\begin{equation*} 
D(z_1,z_2)= \sum_{d,\delta,b_1,b_2=1}^{\infty} 
\frac{d\delta^2b_1b_2\mu(\delta)}{(d\delta b_1)^{z_1}(d\delta b_2)^{z_2}} = \sum_{d=1}^{\infty} \frac1{d^{z_1+z_2-1}}
\sum_{\delta=1}^{\infty} \frac{\mu(\delta)}{\delta^{z_1+z_2-2}}\sum_{b_1=1}^{\infty} \frac1{b_1^{z_1-1}} \sum_{b_2=1}^{\infty} \frac1{b_2^{z_2-1}},
\end{equation*}
giving \eqref{series_lcm_2}.

Now consider the functions $h(n)=\sum_{d\mid n} d\mu(d)= \prod_{p\mid n} (1-p)$, and 
\begin{equation} \label{def_h}
h(n_1,n_2)= \begin{cases} nh(n), & \text{ if $n_1=n_2=n$}, \\ 0, & \text{ otherwise},
\end{cases}
\end{equation}
satisfying
\begin{equation*}
\sum_{n_1,n_2=1}^{\infty} \frac{h(n_1,n_2)}{n_1^{z_1}n_2^{z_2}} = \frac{\zeta(z_1+z_2-1)}{\zeta(z_1+z_2-2)}. 
\end{equation*}

This shows that in the above proof of Theorem \ref{Th_1_LCM}, in the case $k=2$ and $f(n)=n$ one has $h_{f,2}(n_1,n_2)=h(n_1,n_2)$, 
defined by \eqref{def_h}. Now, following that proof,
\begin{equation*}
\sum_{d_1,d_2\le \sqrt{x}} \frac{|h(d_1,d_2)|(d_1+d_2)}{d_1^2d_2^2}= 2 \sum_{d\le \sqrt{x}} \frac{|h(d)|}{d^2}\ll \sum_{d\le \sqrt{x}} \frac1{d}
\ll \log x,
\end{equation*}
and 
\begin{equation*}
\sideset{}{'} \sum_{\substack{d_1,d_2\\ d_1>\sqrt{x}}} \frac{|h(d_1,d_2)|}{d_1^2d_2^2} 
= \sum_{d>\sqrt{x}} \frac{|h(d)|}{d^3} \ll \sum_{d>\sqrt{x}} \frac1{d^2} 
\ll \frac1{\sqrt{x}},
\end{equation*}
leading to the stated error term.
\end{proof}

\begin{proof}[Proof of Theorem {\rm \ref{Th_1_LCM__r_zero}}] Similar to the proof of Theorem \ref{Th_1_LCM}. 
Let $f$ be a function in class ${\cal A}_0$. From Lemma \ref{Lemma_HilTot} with $r=0$ we deduce the identity
\begin{equation*}
f([n_1,\ldots,n_k]) =\sum_{j_1 d_1=n_1,\ldots,j_k d_k=n_k} h_{f,k}(d_1,\ldots,d_k).
\end{equation*}

Hence
\begin{equation*}
\sum_{\substack{n_1,\ldots,n_k\in \N\\ n_1^2+\cdots +n_k^2 \le x}} 
f([n_1,\ldots,n_k]) = 
\sum_{\substack{j_1,d_1,\ldots,j_k,d_k\in \N\\ j_1^2d_1^2+\cdots +j_k^2d_k^2\le x}} 
h_{f,k}(d_1,\ldots,d_k)
\end{equation*}
\begin{equation*}
= \sum_{1\le d_1,\ldots,d_k \le \sqrt{x}} h_{f,k}(d_1,\ldots,d_k) 
\sum_{\substack{j_1,\ldots,j_k\in \N\\ d_1^2j_1^2+\cdots +d_k^2j_k^2 \le x}} 1,
\end{equation*}
and use Lemma \ref{Lemma_number_k}.
\end{proof}

\medskip \medskip

\noindent Randell Heyman \\
School of Mathematics and Statistics \\
University of New South Wales \\
Sydney, Australia \\
E-mail: {\tt randell@unsw.edu.au}

\medskip

\noindent L\'aszl\'o T\'oth \\
Department of Mathematics \\
University of P\'ecs \\
Ifj\'us\'ag \'utja 6, 7624 P\'ecs, Hungary \\
E-mail: {\tt ltoth@gamma.ttk.pte.hu}

\end{document}